\documentclass[11pt]{amsart}
\usepackage{amssymb,euscript,tikz,units}
\usepackage[colorlinks,citecolor=blue,linkcolor=red]{hyperref}
\usepackage{verbatim}
\usepackage[nameinlink]{cleveref}
\usepackage{colonequals}
\usepackage{tikz}

\newcommand{\sM}{\mathcal{M}}
\newcommand{\sA}{\mathcal{A}}
\newcommand{\sB}{\mathcal{B}}
\newcommand{\sL}{\mathcal{L}}
\newcommand{\R}{\mathbb{R}}
\newcommand{\Z}{\mathbb{Z}}
\renewcommand{\H}{\mathbb{H}}
\newcommand{\abs}[1]{\lvert #1 \rvert}

\newcommand{\eps}{\varepsilon}

\newcommand{\vph}{\varphi}
\newcommand{\sbs}{\subset}

\def\length{\mathop{\rm length}}
\def\sys{\mathop{\rm sys}}

\def\arcsinh{\mathop{\rm arcsinh}}

\def\Vol{\mathop{\rm Vol}}
\def\dist{\mathop{\rm dist}}

\def\inj{\mathop{\rm inj}}
\def\Osc{\mathop{\rm Osc}}

\theoremstyle{plain}
\newtheorem{theorem}{Theorem}

\newtheorem{proposition}[theorem]{Proposition}
\newtheorem{lemma}[theorem]{Lemma}

\newtheorem{remark}[theorem]{Remark}

\newcommand{\be}{\begin{equation}}
\newcommand{\ene}{\end{equation}}
\newcommand{\br}{\begin{remark}}
\newcommand{\er}{\end{remark}}
\newcommand{\bl}{\begin{lem}}
\newcommand{\el}{\end{lem}}
\newcommand{\bcor}{\begin{cor}}
\newcommand{\ecor}{\end{cor}}
\newcommand{\bpro}{\begin{pro}}
\newcommand{\epro}{\end{pro}}
\newcommand{\ben}{\begin{enumerate}}
\newcommand{\een}{\end{enumerate}}
\newcommand{\bp}{\begin{proof}}
\newcommand{\ep}{\end{proof}}
\newcommand{\bpo}{\begin{pro}}
\newcommand{\epo}{\end{pro}}
\newcommand{\beq}{\begin{equation*}}
\newcommand{\eeq}{\end{equation*}}
\newcommand{\bear}{\begin{eqnarray}}
\newcommand{\eear}{\end{eqnarray}}
\newcommand{\beqar}{\begin{eqnarray*}}
\newcommand{\eeqar}{\end{eqnarray*}}
\newcommand{\bt}{\begin{theorem}}
\newcommand{\et}{\end{theorem}}
\newcommand{\bex}{\begin{excer}}
\newcommand{\eex}{\end{excer}}

\theoremstyle{definition}

\newtheorem{definition}[theorem]{Definition}

\theoremstyle{remark}

\newtheorem*{rem*}{Remark}
\newtheorem*{ques*}{Question}
\newtheorem*{definition*}{Definition}

\newcommand{\RS}{Riemann surface}

\begin{document}

\title[Optimal lower bound on $\lambda_1$ for large genus]{Optimal lower bounds for first eigenvalues of Riemann surfaces for large genus}

\author{Yunhui Wu and Yuhao Xue}
\address{Yau Mathematical Sciences Center, Tsinghua University, Haidian District, Beijing 100084, China}
\email[(Y.~W.)]{yunhui\_wu@mail.tsinghua.edu.cn}
\email[(Y.~X.) \ ]{xueyh18@mails.tsinghua.edu.cn}

\begin{abstract}
In this article we study the first eigenvalues of closed \RS s for large genus. We show that for every closed \RS \ $X_g$ of genus $g$ $(g\geq 2)$, the first eigenvalue of $X_g$ is greater than $\frac{\sL_1(X_g)}{g^2}$ up to a uniform positive constant multiplication. Where $\sL_1(X_g)$ is the shortest length of multi closed curves separating $X_g$. Moreover,we also show that this new lower bound is optimal as $g \to \infty$. 
\end{abstract}

\maketitle

\section{Introduction}
For any integer $g\geq 2$, let $\sM_g$ be the moduli space of closed \RS s of genus $g$ and $X_g\in \sM_g$ be a closed hyperbolic surface of genus $g$. The spectrum of the Laplacian on $X_g$ is a fascinating topic in several mathematical fields including analysis, geometry, number theory, topology and so on for a long time. The spectrum of $X_g$ is a discrete closed subset in $\R^{\geq 0}$ and consists of eigenvalues with finite multiplicity. We enumerate them, counted with multiplicity, in the following increasing order
\[0=\lambda_0(X_g)<\lambda_1(X_g)\leq \lambda_2(X_g) \leq \cdots.\]

\noindent Buser \cite{Bus77} showed that for any constant $\eps>0$, there exists a hyperbolic surface $\mathcal X_g$ of genus $g$ such that $\lambda_{2g-3}(\mathcal X_g)<\eps$ and $\lambda_{n}(\mathcal X_g)<\frac{1}{4}+\eps$ for any $n\geq (2g-2)$. Recently Otal-Rosas \cite{OR09} showed that $\lambda_{2g-2}(X_g)>\frac{1}{4}$ for any $X_g\in \sM_g$. One may also see Ballmann-Matthiesen-Mondal \cite{BMM16,BMM17} and Mondal \cite{Mon14} for more general statements on $\lambda_{2g-2}(X_g)$.

\begin{definition*}
For any $X_g \in \sM_g$ and integer $i\in [1,2g-3]$, we define a positive quantity $\sL_i(X_g)$ of $X_g$ to be minimal possible sum of the lengths of simple closed geodesics in $X_g$ which cut $X_g$ into $i+1$ pieces.
\end{definition*}

 The quantity $\sL_i(X_g)$ can be arbitrarily closed to $0$ for certain $X_g \in \sM_g$. Schoen-Wolpert-Yau \cite{SWY80} showed that the $i$-th eigenvalue of $X_g$ is comparable to the quantity $\sL_i(X_g)$ of $X_g$ above. More precisely, they showed that for any integer $i \in [1, 2g-3]$, there exists two constants $\alpha_i(g)>0$ and $\beta_i(g)>0$, depending on $g$, such that for any $X_g\in \sM_g$, \bear \label{SWY-eigen} \alpha_i(g)\leq \frac{\lambda_i(X_g)}{\sL_i(X_g)}\leq \beta_i(g).\eear
One may see Dodziuk-Randol \cite{DR86} for a different proof of Schoen-Wolpert-Yau's theorem, and see Dodziuk-Pignataro-Randol-Sullivan \cite{DPRS85} on similar results for \RS s with punctures. The upper bounds in \eqref{SWY-eigen} follow by suitable choices of test functions on certain collars, whose central closed geodesics are part of a pants decomposition of $X_g$ whose boundary curves have bounded lengths in terms of $g$, which is due to Bers \cite{Bers-c}. For proving the lower bounds in \eqref{SWY-eigen}, one essential step is to show that $\lambda_1(X_g)\geq \alpha_1(g) \cdot \sL_1(X_g)$. Which was applied in \cite{DR86} to obtain the lower bounds in \eqref{SWY-eigen} for other eigenvalues $\lambda_i(X_g)\ (2\leq i \leq 2g-3)$, together by using a mini-max principle. 

In this paper we study the asymptotic behavior of the constant $\alpha_1(g)$ in \eqref{SWY-eigen} for large genus. The method in this article is motivated by \cite{DR86} of Dodziuk-Randol. We prove 
\bt \label{mt-1}
For every $g\geq 2$, there exists a uniform constant $K_1>0$ independent of $g$ such that for any hyperbolic surface $X_g\in \sM_g$, the first eigenvalue $\lambda_1(X_g)$ of $X_g$ satisfies that
\[\lambda_1(X_g)\geq K_1\cdot \frac{\mathcal{L}_1(X_g)}{g^2}.\]
\et

\noindent By \cite{SWY80} we know that Theorem \ref{mt-1} is optimal as $\mathcal{L}_1(X_g)\to 0$. Actually Theorem \ref{mt-1} is also optimal as $g\to \infty$: in \cite{WX18} we constructed a hyperbolic surface $\mathcal{X}_g\in \sM_g$ for all $g\geq 2$ such that the first eigenvalue satisfies that
\[\lambda_1(\mathcal{X}_g)\leq K_2 \cdot \frac{\mathcal{L}_1(\mathcal{X}_g)}{g^2}\]
where $K_2>0$ is a uniform constant independent of $g$. In Section \ref{sec-opti} we will discuss the details to see that Theorem \ref{mt-1} is optimal simultaneously as $\mathcal{L}_1(X_g)\to 0$ and $g \to \infty$.

\subsection*{One open question} It is interesting to study the optimal asymptotic behaviors of the other constants $\alpha_i(g) \ (2\leq i \leq 2g-3)$ in Schoen-Wolpert-Yau's theorem as $g\to \infty$. The answer to the following question is \emph{unknown}.

\begin{ques*}\label{ulb-eig-lg}
For every $g\geq 2$, does there exist a uniform constant $K>0$ independent of $g$ such that for any hyperbolic surface $X_g\in \sM_g$ and $i\in [1,2g-3]$, the $i$-th eigenvalue $\lambda_i(X_g)$ of $X_g$ satisfies that
\[\lambda_i(X_g)\geq K\cdot \frac{i\cdot \mathcal{L}_i(X_g)}{g^2}?\]
\end{ques*}

\begin{rem*}
Theorem \ref{mt-1} affirmatively answers the question above for case $i=1$.
\end{rem*}

\subsection*{Plan of the paper.}
The paper is organized as follows. In Section \ref{sec-prel} we collect preliminaries for closed hyperbolic surfaces and provide several necessary technical lemmas. In Section \ref{sec-gap} we follow \cite{DR86} to prove that if the first eigenvalue $\lambda_1(X_g)$ is small, then there exists a uniform gap in terms of $g$ between the oscillation of the first eigenfunction on certain modified thick part and the total oscillations of the first eigenfunction on all components of this modified thick part, which is Proposition \ref{eige-gap}. For this reason, the first eigenfunction has large oscillation on the modified thin part consisting of certain collars. By definition the first eigenvalue is greater than the energy of its eigenfunction on the modified thin part. Then in Section \ref{sec-mt-1} by using some combinatorial discussion and estimations in previous sections, we finish the proof of Theorem \ref{mt-1}. In the last Section we discuss an example to see that the lower bound in Theorem \ref{mt-1} is optimal as $g$ goes to $\infty$.  

\subsection*{Acknowledgements.}
The authors would like to thank Long Jin for helpful discussions. The first named author is partially supported by a grant from Tsinghua University. We are grateful to one referee to point out a mistake in our original statement of Lemma \ref{sign change}. We are also grateful to another referee for useful comments and suggestions which are helpful.

%%%%%%%%%%%%%%%%%%%%%%%%%%%%%%%%%%%%%
\section{Preliminaries}\label{sec-prel}
In this section we will set up the notations and provide some necessary background on two-dimensional hyperbolic geometry and spectrum theory of hyperbolic surfaces.
%%%%%%%%%%%%%%%
\subsection{Hyperbolic surfaces}
Let $X_g$ be a closed hyperbolic surface of genus $g\geq 2$ and $\gamma \subset X_g$ be a non-trivial loop. There always exists a unique closed geodesic, still denoted by $\gamma$, representing this loop. The Collar Lemma says that it has a tubular neighborhood which is a topological cylinder with a standard hyperbolic metric. And the width of this cylinder, only depending on the length of $\gamma$, goes to infinity as the length of $\gamma$ goes to $0$. We use the following version \cite[Theorem 4.1.1]{Buser10} of the Collar Lemma.

\begin{lemma}[Collar lemma]\label{collar}
	Let $\gamma_1 , \gamma_2, ..., \gamma_m$ be disjoint simple closed geodesics on a closed hyperbolic Riemann surface $X_g$, and $\ell(\gamma_i)$ be the length of $\gamma_i$. Then $m\leq 3g-3$ and we can define the collar of $\gamma_i$ by
	$$T(\gamma_i)=\{x\in X_g; \ \dist(x,\gamma_i)\leq w(\gamma_i)\}$$
	where
	$$w(\gamma_i)=\mathop{\rm arcsinh} \frac{1}{\sinh \frac{1}{2}\ell(\gamma_i)}$$
	is the half width of the collar.
	
	Then the collars are pairwise disjoint for $i=1,...,m$. Each $T(\gamma_i)$ is isomorphic to a cylinder $(\rho,t)\in [-w(\gamma_i),w(\gamma_i)] \times \mathbb S ^1$, where $\mathbb S ^1 = \R / \Z$, with the metric
\bear\label{collar-metric}	
ds^2=d\rho^2 + \ell(\gamma_i)^2 \cosh^2\rho dt^2.
\eear
	And for a point $(\rho,t)$, the point $(0,t)$ is its projection on the geodesic $\gamma_i$, $\abs{\rho}$ is the distance to $\gamma_i$, $t$ is the coordinate on $\gamma_i \cong \mathbb S ^1$.
\end{lemma}
\noindent As the length $\ell(\gamma)$ of the central closed geodesic goes to $0$, the width \bear \label{wid-large} e^{w(\gamma)} \sim \frac{4}{\ell(\gamma)}\eear which tends to infinity. In this paper, we mainly deal with the case that $\ell(\gamma)$ is small and so $w(\gamma)$ is large.

Note that a collar is homeomorphic to a cylinder which may have arbitrary large width. When saying a collar in this paper, we will not always assume it has the maximal width given in Lemma \ref{collar}. Actually we will use a subcollar with a slightly shorter width. Now we make certain elementary computations on a collar. Assume $\gamma \sbs X_g$ is a closed geodesic of length $\ell>0$. There are two types of coordinates on a collar of $\gamma$. One is $(\rho,t)$ given in Lemma \ref{collar}. The other one is the polar coordinate $(r,\theta)$ on the upper half plane $\mathbb H$. Which may be written as
\begin{equation}\label{coord trans}
\begin{cases}
\rho &= -\log \tan \frac{\theta}{2} \\
t &= \frac{1}{\ell} \log r.
\end{cases}
\end{equation}

\noindent For any two points $z$ and $w\in\mathbb H$, the hyperbolic distance $\dist_\H(z,w)$ satisfies that
\begin{equation}\label{hyper dist}
\cosh \text{dist}_{\mathbb H} (z,w) = 1+ \frac{|z-w|^2}{2\mathop{\rm Im} z \mathop{\rm Im} w}.
\end{equation}

\noindent For any two points $(\rho,0)$ and $(\rho,t)$ in the coordinate given in Lemma \ref{collar} which have the same distance to the center closed geodesic $\gamma$, the curve $s\mapsto (\rho,s)$ where $0\leq s \leq t$ is not a geodesic. However one may compute the length $L((\rho,0),(\rho,t))$ of the geodesic homotopic to that curve (see figure \ref{collar1}) by using (\ref{hyper dist}). Actually we have
\begin{equation}
\cosh L((\rho,0),(\rho,t)) = \cosh \text{dist}_{\mathbb H} (e^{i\theta},re^{i\theta}) = 1+ \frac{|r-1|^2}{2r\sin^2\theta} \nonumber
\end{equation}
where $(r,\theta)$ is given in \eqref{coord trans}.
\begin{figure}[ht]
	\begin{center}
		\begin{tikzpicture}[scale=0.8]
		
		\draw (0,0) +(0,0.5) ..controls+(-2,0)and+(0.8,-0.5).. +(-7,2)
		+(0,0.5) ..controls+(2,0)and+(-0.8,-0.5).. +(7,2)
		+(0,-0.5) ..controls+(-2,0)and+(0.8,0.5).. +(-7,-2)
		+(0,-0.5) ..controls+(2,0)and+(-0.8,0.5).. +(7,-2);
		
		\draw (0,0) +(0,0.5) ..controls+(-0.1,-0.5).. +(0,-0.5);
		\draw[dashed] (0,0) +(0,0.5) ..controls+(0.1,-0.5).. +(0,-0.5);
		
		\draw[->] (-2,1.8)--(2,1.8);
		\draw (0,2.1) node {$\rho$};
		
		\draw[->] (-0.2,-0.4) ..controls+(-0.1,0.4).. (-0.2,0.4);
		\draw (-0.7,0) node {$t$};
		
		\draw (4.1,0) +(0,1.1) ..controls+(-0.2,-1.1).. +(0,-1.1);
		\draw[dashed] (4.1,0) +(0,1.1) ..controls+(0.2,-1.1).. +(0,-1.1);
		
		\filldraw (4,-0.7)circle[radius=0.03];
		\filldraw (4,0.7)circle[radius=0.03];
		
		\draw (4,0) +(0,0.7) ..controls+(-0.5,-0.2)and+(-0.5,0.2).. +(0,-0.7);
		\draw (4.6,-0.7) node {$(\rho,0)$};
		\draw (4.6,0.7) node {$(\rho,t)$};
		
		\draw[->] (3.6,-0.3)--(2,-1.2);
		\draw (2,-1.5) node {geodesic};
		
		\draw[->] (4,0)..controls+(2,-0.2)and+(1,1).. (5,-1.9);
		\draw (4.3,-2.2) node {the curve $s \mapsto (\rho,s)$};
		
		\end{tikzpicture}
	\end{center}
	\caption{} \label{collar1}
\end{figure}

\noindent By applying \eqref{coord trans} we have
\begin{equation}\label{dist collar}
\sinh \frac{L((\rho,0),(\rho,t))}{2} = \sinh \frac{t\ell}{2} \cosh \rho.
\end{equation}
In particular, the injectivity radius $\inj(\rho,0))$ at the point $(\rho,0)$ satisfies
\begin{equation}\label{inj collar}
\sinh \inj ((\rho,0)) = \sinh \frac{L((\rho,0),(\rho,1))}{2} = \sinh \frac{\ell}{2} \cosh \rho.
\end{equation}

Let $[-w,w]\times \mathbb S^1$ endowed with the hyperbolic metric given in \eqref{collar-metric}. We assume that the center closed geodesic has length of $\ell$ and the width is $w$ which may not be the maximal width given in Lemma \ref{collar}. Then the hyperbolic volume of the collar $[-w,w]\times \mathbb S^1$ is
\begin{align} \label{area collar}
	\Vol([-w,w]\times \mathbb S^1)
	&= \int_0^1 \int_{-w}^w \ell\cosh \rho d\rho dt  \\
	&= 2\ell\sinh w.\nonumber
\end{align}

One may refer to \cite{Buser10} for more details in this subsection.

%%%%%%%%%%%%%%%%%%%%%%%%
\subsection{Thick-thin decomposition}\label{thick-thin decomp}
In this subsection, we make a thick-thin part decomposition of $X_g\in \sM_g$. For any $p \in X_g$ we let $\inj(p)$ denote the injectivity radius of $X_g$ at $p$. For any given constant $\eps>0$, we define
$$X_g^{\geq \eps}:=\{p\in X_g; \ \inj(p)\geq\eps \}$$
to be the $\eps$-thick part of $X_g$, and its complement
$$X_g^{<\eps}:=\{p\in X_g; \ \inj(p)<\eps \}$$
to be the $\eps$-thin part of $X_g$. By the Collar Lemma \ref{collar}, we know that for a small $\eps>0$, the thin part $X_g^{<\eps}$ consists of certain disjoint collars (or empty).

In this paper we use a modified thick-thin decomposition. More precisely, for a small enough given constant $\eps>0$ (given in Lemma \ref{thick-thin}), we consider all closed geodesics of length less than $2\eps$ which are denoted by $\gamma_1,...,\gamma_m$ for some $m\geq 0\in \Z$. Consider the collar $T_i$ of each $\gamma_i$ defined by
\begin{equation}
T_i := \{x \in X_g; \ \dist(x,\gamma_i) \leq w_i\}
\end{equation}
where $\ell_i$ is the length of $\gamma_i$ and $w_i$ is the width of $T_i$ given by
\begin{equation}\label{w_i}
0\leq w_i = \max\{0, \arcsinh \frac{1}{\sinh \frac{1}{2}\ell_i} - 2\}.
\end{equation}
Here the width $w_i$ is less that the maximal width $w(\gamma_i)$ in Lemma \ref{collar}. Thus, for small enough $\eps>0$, $\{T_i\}_{1\leq i \leq m}$ are also pairwisely disjoint.

\begin{definition}\label{def thick-thin}
	We define
	\begin{equation}\label{def-sB}
	\mathcal B: = \bigcup_{i=1}^m T_i.
	\end{equation}
which is called the \emph{$\eps$-modified thin part} of $X_g$. And we define
	\begin{equation}\label{def-sA}
	\mathcal A: = \overline{X_g \setminus \mathcal B}
	\end{equation}
which is called the \emph{$\eps$-modified thick part} of $X_g$. Furthermore, for each $\gamma_i$, we define
	\begin{equation}\label{def-S}
	S_i: = \{x \in X_g; \ w_i \leq \dist(x,\gamma_i) \leq w_i +1\} 
	\end{equation}
which is called the \emph{shell} of $T_i$.
\end{definition}

\noindent For the modified $\eps$-thick-thin decomposition above, we have the following properties which will be applied in the proof of Proposition \ref{oscillation} and \ref{max thick part} to deal with some technical details. They just follow from some elementary computations on collars.

\begin{lemma}\label{thick-thin}
	There exists a uniform constant $\eps>0$ independent of $g$ such that for every hyperbolic surface $X_g$ of genus $g$, the modified $\eps$-thick-thin decomposition $X_g=\mathcal A \cup \mathcal B$ satisfies the following properties. 
\ben
\item The width of a collar given by (\ref{w_i}) satisfies
	\begin{equation}
	w_i\geq 1 > \eps. \nonumber
	\end{equation}
\item The closed geodesics $\gamma_1,...,\gamma_m\in \sB$ are disjoint. The corresponding $T_1 \cup S_1,...,T_m \cup S_m$ are also disjoint. Moreover, for each $i  \in [1,m]$, $S_i \sbs \mathcal A$.

\item The volumes of all collars $T_i$ and shells $S_i$ are bounded. More precisely,
	\begin{equation}
	\frac{1}{2}\leq \Vol(T_i) \leq 4 \ \ \ \text{and}\ \ \ \frac{1}{2}\leq \Vol(S_i) \leq 4. \nonumber
	\end{equation}
\item For each point $p\in \mathcal A$,
	\begin{equation}
	\inj (p) \geq \eps  .\nonumber
	\end{equation}
\item For any $\eta \in (0, \eps]$, if two points $p$ and $q$ in one component of $\mathcal A$ have distance $\dist(p,q) = \eta$ in $X_g$, then there exists a path connecting $p$ and $q$ which is contained in the component such that it has length less than $5\eta$.

\een
\end{lemma}

\begin{proof}
	(i) Recall that the length $\ell_i$ of the closed geodesic $\gamma_i$ is less than $2\eps$. The width $w_i$ is given by (\ref{w_i}). So we only need $\eps>0$ to be small enough such that
	\begin{equation}\label{eps (i)}
	\arcsinh \frac{1}{\sinh \eps} - 2 \geq 1 > \eps.
	\end{equation}\

	(ii) By Lemma \ref{collar}, for $\eps < \arcsinh 1$  the closed geodesics $\{\gamma_1,...,\gamma_m\}$ are disjoint. By definition we know that for each $1\leq i \leq m$, \[T_i \cup S_i=\{x\in X_g; \ \dist(x, \gamma_i)\leq \arcsinh (\frac{1}{\sinh \frac{1}{2}\ell_i}) - 1\}\]
which is contained in the collar of center closed geodesic $\gamma_i$ with maximal width $\arcsinh (\frac{1}{\sinh \frac{1}{2}\ell_i})$. By Lemma \ref{collar} we know that $\{T_1 \cup S_1,...,T_m \cup S_m\}$ are also disjoint. By definition we clearly have that $S_i \sbs \mathcal A$ for each $i  \in [1,m]$.\\

	(iii) By (\ref{area collar}), we have
	\begin{equation}
	\Vol (T_i) = 2\ell_i \sinh w_i \nonumber
	\end{equation}
and
	\begin{equation}
	\Vol (S_i) = 2\ell_i \sinh (w_i +1) - 2\ell_i \sinh w_i. \nonumber
	\end{equation}
	So we have
	\begin{equation}
	\Vol (T_i) + \Vol (S_i) = 2\ell_i \sinh (w_i +1) \leq \frac{2\ell_i}{\sinh \frac{1}{2}\ell_i} \leq 4. \nonumber
	\end{equation}
	As functions of $\ell_i$, by using \eqref{wid-large} we have
	\begin{equation}
	\lim_{\ell_i \to 0}\Vol (T_i) = \lim_{\ell_i \to 0}2\ell_i \cdot \sinh( \log (\frac{4}{\ell_i})-2)= \frac{4}{e^2}, \nonumber
	\end{equation}
and
	\begin{equation}
\lim_{\ell_i \to 0} \Vol (S_i) = \frac{4(e-1)}{e^2}. \nonumber
	\end{equation}
	Thus, for a small enough constant $\eps>0$ we have
	\begin{equation}\label{eps (iii)}
	\frac{1}{2}\leq \Vol (T_i) \leq 4 \ \ \text{and} \ \ \frac{1}{2}\leq \Vol (S_i) \leq 4.
	\end{equation}\

	(iv) For a point $p\in \mathcal A$, suppose that $\inj (p) < \eps$. Then there exists a geodesic loop $\alpha$ based at $p$ with length $\ell(\alpha) < 2\eps$. This closed curve $\alpha$ should be homotopic to a closed geodesic with length less than $2\eps$, so it is one of the $\gamma_1,...,\gamma_m$ denoted by $\gamma_i$. In the collar of such a geodesic $\gamma_i$, one may assume the coordinate of $p$ is $(\rho,0)$ with $\rho \geq w_i$. For small enough $\eps>0$ satisfying $$\sinh(\eps)\leq 2\eps,$$ by \eqref{inj collar} we have
\bear
 2\eps &>& \ell(\alpha) \geq \sinh (\frac{1}{2}\ell(\alpha))  \geq \sinh(\inj(p)) \geq \sinh\frac{\ell_i}{2} \cosh \rho\\
		&\geq& \sinh\frac{\ell_i}{2} \cosh w_i \geq \sinh\frac{\ell_i}{2} \cosh ( \arcsinh (\frac{1}{\sinh \frac{1}{2}\ell_i})) \cdot e^{-2}  \nonumber\\
		&\geq& \frac{1}{e^2} \nonumber
\eear
where we apply the inequality $\cosh (x-y)\geq \cosh(x)e^{-y}$ for all $x,y\geq 0$. So if
	\begin{equation}\label{eps (iv)}
	\eps < \frac{1}{2e^2},
	\end{equation}
	then $\inj (p) \geq \eps$ for all $p\in \mathcal A$. Remark here it is easy to see that $\sinh(\eps)\leq 2\eps$ if $0<\eps < \frac{1}{2e^2}$.\\

	(v) Assume that $p$ and $ q$ are in one component of $\mathcal A$ which have distance $$\dist(p,q) = \eta\leq \eps.$$ If the shortest geodesic connecting $p$ and $q$ lies in $\mathcal A$, then the distance from $p$ to $q$ in $\mathcal A$ is exactly $\eta$. We are done in this case. So one may assume that the shortest geodesic $\alpha$ connecting $p$ and $q$ crosses the boundaries of $\mathcal A$. Now we assume $\eps>0$ satisfies $(i)-(iv)$. In particular, $\eps < \frac{1}{2e^2}$. Since $\dist(p,q)\leq \eps<1$, these two points $p$ and $q$ must lie in one certain shell $S_i$ for some $i\in [1,m]$. Assume in the coordinate on collar $p=(\rho_1,t_1)$, $q=(\rho_2, t_2)$ and $0< \rho_1 \leq \rho_2$. Let $r=(\rho_1,t_2)\in S_i \subset \sA$. We Consider the curve
	$$\alpha_2 : s\mapsto (s,t_2)$$
	connecting $r$ and $q$ in $\mathcal A$ and the curve
	$$\beta : s\mapsto (\rho_1,s)$$
	connecting $p$ and $r$ in $\mathcal A$ such that $\beta \cup \alpha_2$ is homotopic to $\alpha$. By the structure of collar and shell, $\alpha_2$ is a shortest geodesic but $\beta$ is not a geodesic. Consider the shortest geodesic $\alpha_1$ connecting $p$ and $r$ which is homotopic to $\beta$.
	
	\begin{figure}[ht]
		\begin{center}
			\begin{tikzpicture}[scale=0.8]
			
			\draw (0,0) +(-1,2) ..controls+(2,0.1)and+(-2,-0.5).. +(6,3)
			+(-1,-2) ..controls+(2,-0.1)and+(-2,0.5).. +(6,-3);
			
			\draw (0,0) +(0,2.1) ..controls(-0.2,0).. +(0,-2.1);
			\draw[dashed] (0,0) +(0,2.1) ..controls(0.2,0).. +(0,-2.1);
			
			\draw (5,0) +(0,2.8) ..controls(4.8,0).. +(0,-2.8);
			\draw[dashed] (5,0) +(0,2.8) ..controls(5.2,0).. +(0,-2.8);
			
			\draw (0,2.4) node {$\rho = w_i$};
			\draw (5,3.3) node {$\rho = w_i+1$};

			\filldraw (1.5,-1.3)circle[radius=0.03];
			\draw (1.5,-1.6) node {$p$};
			\filldraw (1.5,1.3)circle[radius=0.03];
			\draw (1.5,1.55) node {$r$};
			\filldraw (3.5,1.5)circle[radius=0.03];
			\draw (3.8,1.5) node {$q$};

			\draw (1.5,0) +(0,1.3) ..controls(1.3,0).. +(0,-1.3);
			\draw (1.6,0) node {$\beta$};
			
			\draw (1.5,0) +(0,1.3) ..controls+(-5.8,-0.3)and+(-0.8,0.3).. +(0,-1.3);
			\draw (-1.6,0.5) node {$\alpha_1$};		

			\draw (1.5,1.3) -- (3.5,1.5);
			\draw (2.5,1.6) node {$\alpha_2$};
			
			\draw (1.5,-1.3) ..controls+(-3,0.2)and+(-6,-1.5).. (3.5,1.5);
			\draw (-1.1,-0.5) node {$\alpha$};
			
			\end{tikzpicture}
		\end{center}
		\caption{} \label{shell1}
	\end{figure}

\noindent The geodesic $\alpha_2$ achieves the distance between $\{(\rho,t)\in S_i ;\ \rho= \rho_1 \}$ and$\{(\rho,t)\in S_i ;\ \rho= \rho_2 \}$. So we have
	\begin{equation}
	\ell(\alpha_2) \leq \dist(p,q) = \eta. \nonumber
	\end{equation}
	Since $\alpha_1$ is homotopic to $\alpha \cup \alpha_2^{-1}$ and $\alpha_1$ is a shortest geodesic, by the triangle inequality we have
	\begin{equation}
	\ell(\alpha_1) \leq \ell(\alpha)+\ell(\alpha_2) \leq 2\eta. \nonumber
	\end{equation}
	By (\ref{dist collar}),
	\begin{equation}
	\sinh \frac{\ell(\alpha_1)}{2} = \sinh\frac{|t_1-t_2|\ell_i}{2}\cosh \rho_1. \nonumber
	\end{equation}
	If $\eps>0$ is small enough and satisfies
	\begin{equation}\label{eps (v)}
	 \sinh\eps \leq 2\eps, \nonumber
	\end{equation}
	then
	\begin{equation}
	\ell(\alpha_1) \geq \frac{|t_1-t_2|\ell_i}{2}\cosh \rho_1 \nonumber
	\end{equation}
	because $\frac{\ell(\alpha_1)}{2} \leq \eta \leq \eps$ and $\frac{|t_1-t_2|\ell_i}{2} \leq \frac{\ell_i}{2} \leq \eps$. On the other hand,
	\begin{align}
		\ell(\beta)
		&=|\int_{t_1}^{t_2} \ell_i\cosh \rho_1 ds | \nonumber \\
		&=|t_1-t_2|\ell_i\cosh \rho_1. \nonumber
	\end{align}
	So we have
	\begin{equation}
	\ell(\beta) \leq 2\ell(\alpha_1). \nonumber
	\end{equation}
	Clearly $\beta\cup\alpha_2\subset \sA$ is a path connecting $p$ and $q$ whose length satisfies that
	\begin{align}
		\ell(\beta) + \ell(\alpha_2)
		\leq 2\ell(\alpha_1) + \eta \leq 5\eta. \nonumber 
	\end{align}
	Which completes the proof.
\end{proof}

\begin{rem*} \ben
\item	The uniform constant $\eps>0$ only needs to satisfy (\ref{eps (i)}), (\ref{eps (iii)}) and (\ref{eps (iv)}). For example, one may choose
	$$\eps=0.05.$$
\item The bounds in this lemma are not optimal. But they are good enough to be applied to prove our later propositions.  	
\een
\end{rem*}

%%%%%%%%%%%%%%%%
\subsection{An upper bound for $\sL_i(X_g)$ $(1\leq i \leq (2g-3))$} Recall that for all $g\geq 2$ the Bers' constant $L_g$ is the best possible constant such that for each hyperbolic surface $X_g$ of genus $g$ there exists a pants decomposition whose boundary curves have length less than $L_g$. One may see \cite[Chapter 5]{Buser10} for more details. The following result is due to Bers.
\bt[Bers, \cite{Bers-c}]\label{Bers-c}
For each $g\geq 2$, \[L_g\leq 26(g-1).\]
\et
\noindent In this subsection we apply Theorem \ref{Bers-c} to get an upper bound for $\sL_1(X_g)$ which will be applied to prove Theorem \ref{mt-1}. More precisely, we show
\begin{lemma}\label{ub-L1}
For every hyperbolic surface $X_g$ of genus $g$ and $i\in [1,2g-3]$,
\[\sL_i(X_g)\leq 78\cdot i\cdot (g-1).\]
\end{lemma}
\bp
Let $\{\mathcal{P}_j\}_{1\leq j \leq (2g-2)}$ be a pants decomposition of $X_g$, where each $\mathcal{P}_j$ is a pair of pants, such that
\bear\label{ub-L1-1}
\max_{1\leq j \leq (2g-2)}\max_{\alpha \in \partial \mathcal{P}_j}\ell_{\alpha}(X_g)\leq L_g.
\eear
For each $1\leq i \leq (2g-3)$, we set $$S:=\{\alpha;\ \alpha \in \partial \mathcal{P}_j \ \emph{for all $1\leq j \leq i$}\}.$$
It is not hard to see that $X_g \setminus S$ has components containing all the $\mathcal{P}_j$'s where $1\leq j \leq i$ and their complement $X_g\setminus \left(\cup_{1\leq j \leq i} \mathcal{P}_j \right)$. Thus, $S$ divides $X_g$ into at least $(i+1)$ components. By construction we have
\bear\label{ub-L1-2}
\# S \leq 3\cdot i.
\eear
By the definition of $\sL_i(X_g)$ we know that
\bear\label{ub-L1-3}
\sL_i(X_g)\leq \sum_{\alpha \in S}\ell_{\alpha}(X_g).
\eear

Then it follows by \eqref{ub-L1-1}, \eqref{ub-L1-2}, \eqref{ub-L1-3} and Theorem \ref{Bers-c} that
\bear
\sL_i(X_g)\leq 3\cdot i \cdot L_g \leq 78 \cdot i \cdot (g-1).
\eear
Which completes the proof.
\ep

\begin{rem*}
\ben
\item It would be interesting to study the asymptotic behavior of the quantity $\sup_{X_g \in \sM_g}\sL_i(X_g)$, where $1\leq i \leq (2g-3)$, as $g\to \infty$. To our best knowledge, it is even \emph{unkown} for $\sup_{X_g \in \sM_g}\sL_1(X_g)$ as $g\to \infty$. After this article was submitted, recently we show in \cite{NWX20} that $\sup_{X_g \in \sM_g}\sL_1(X_g)\leq C \ln(g)$ for all $g\geq 2$ and some universal constant $C>0$ independent of $g$.

\item It is known \cite[Theorem 5.1.4]{Buser10} that the Bers constant $L_g\geq \sqrt{6g}-2$. As $g\to \infty$, the asymptotic behavior of $L_g$ is still \emph{unknown}. The upper bounds in \eqref{SWY-eigen} of Schoen-Yau-Wolpert depend on $L_g$.  
\een
\end{rem*}

%%%%%%%%%%%%%%%%
\subsection{Eigenvalues}
Let $X_g$ be a closed \RS \ of genus $g\geq 2$ which can also be viewed as a hyperbolic metric on $X_g$. Let $\Delta$ be the Laplacian with respect to this metric. A number $\lambda$ is called an \emph{eigenvalue} if $\Delta f+ \lambda \cdot f=0$ on $X_g$ for some non-zero function $f$ on $X_g$. And the corresponding function $f$ is called an \emph{eigenfunction}. It is known that the set of eigenvalues is an infinite sequence of non-negative numbers
\[0=\lambda_0(X_g)<\lambda_1(X_g)\leq \lambda_2(X_g) \leq \cdots.\]
Let $\{f_i\}_{i \geq 0}$ be its corresponding orthonormal sequence of eigenfunctions. Clearly $f_0$ is the constant function $\frac{1}{\sqrt{4\pi(g-1)}}$. The mini-max principle tells that for any integer $k\geq 0$,
\beqar \label{lamb-k}
\lambda_k(X_g)=\inf_{}\{ \frac{\int_{X_g}|\nabla f|^2}{\int_{X_g}f^2}; \ 0\neq f \in H^1(X_g) \ \text{and} \ \int_{X_g}f\cdot f_i=0 \ \forall i \in [0, k-1]\}
\eeqar
where $H^1(X_g)$ is the completion under the $H^1$-norm of the space of smooth functions on $X_g$. One may see \cite{Chavel} for more details.

%%%%%%%%%%%%%

\subsection{Energy bounds on collars}
In this subsection we recall several useful energy bounds in \cite{DR86} of certain functions on collars. And we will make a little modification for one of them.

The first one is a special case of \cite[Lemma 1]{DR86} for $n=2$.
\begin{lemma}\cite[Lemma 1]{DR86} \label{eigen collar1}
Let $X_g \in \sM_g$ be a hyperbolic surface and $T=[-w,w]\times \mathbb S^1$ be a collar of $X_g$. Let $\lambda_1(T)$ be the first Dirichlet eigenvalue for $T$. Then
$$\lambda_1(T)>\frac{1}{4}.$$
\end{lemma}

The second one is as follows.
\begin{lemma}\cite[Lemma 3]{DR86} \label{eigen collar3}
Let $X_g \in \sM_g$ be a hyperbolic surface, and $T=[-w,w]\times \mathbb S^1$ be a collar of $X_g$ with center closed geodesic of length $\ell$. Denote by $\Gamma_1$ and $\Gamma_2$ the two boundary components of $T$, which are topologically circles. Suppose a smooth function $f$ on $T$ satisfies
\begin{equation}
\min_{(x,x^*)\in \Gamma_1 \times \Gamma_2} |f(x)-f(x^*)| = c \geq 0  \nonumber
\end{equation}
where $x^*\in \Gamma_2$ is the reflection of $x\in \Gamma_1$ through the center closed geodesic. Then
\begin{equation}
\int_T |\nabla f|^2 \geq \frac{c^2}{4} \ell.  \nonumber
\end{equation}
\end{lemma}

The third one is a slightly different version of \cite[Lemma 2]{DR86}.
\begin{lemma}\label{eigen collar2}
Let $X_g \in \sM_g$ be a hyperbolic surface, and $T=[-w,w]\times \mathbb S^1$ be a collar of $X_g$ with shell $S=[-w-1,-w]\times \mathbb S^1 \cup [w,w+1]\times \mathbb S^1$. Let $\delta$ be a constant with $0<\delta <\frac{1}{16}$, $c>0$ be a constant and $f$ be a smooth function on $T\cup S$ satisfying:
\ben
\item  $\int_{T} |f|^2 \geq c >0$,
\item $\int_{S} |f|^2 \leq \delta c$,
\item $\int_{S} |\nabla f|^2 \leq \delta c$.
\een
Then
\begin{equation}
\int_{T} |\nabla f|^2 \geq \frac{1-16\delta}{4} c. \nonumber
\end{equation}
\end{lemma}

\begin{proof}
We follow the argument in the proof of \cite[Lemma 2]{DR86}. First recall that the Collar Lemma \ref{collar} tells that one may assume $(\rho,t)\in [-w-1,w+1]\times \mathbb S^1$ is a coordinate on $T\cup S$. And the hyperbolic metric on $T\cup S$ is
$$ds^2=d\rho^2 + \ell^2 \cosh^2\rho dt^2$$
where $\ell$ is the length of center closed geodesic. In this coordinate we define a function $F$ on $T\cup S$ as
\begin{equation}
F(\rho,t) :=
\begin{cases}
f(\rho,t) & \text{if}\ |\rho|\leq w, \\
(w+1-|\rho|)f(\rho,t) & \text{if}\ |\rho|\geq w.
\end{cases} \nonumber
\end{equation}
By definition $F|_{\partial (T\cup S)} =0$. Then it follows by Lemma \ref{eigen collar1} that
\begin{equation}\label{c-in-1}
\int_{T\cup S}|\nabla F|^2 > \frac{1}{4}\int_{T\cup S}F^2. 
\end{equation}

\noindent It is clear that $|\nabla (w+1-|\rho|)|^2 =|\nabla |\rho||^2= 1$ on $S$. So we have
\bear \label{c-in-2}
\int_S |\nabla F|^2 &=& \int_S |\nabla (w+1-|\rho|) \cdot f(\rho,t)+(w+1-|\rho|) \cdot \nabla f(\rho,t)|^2 \\
&\leq &  \int_S (|f(\rho,t)|+(w+1-|\rho|) \cdot |\nabla f(\rho,t)|)^2 \nonumber\\
&\leq& 2\int_S f^2 + 2\int_S |\nabla f|^2 \quad \emph{(by Cauchy-Schwarz inequality)} \nonumber\\
& \leq& 4\delta c.  \nonumber
\eear
Where assumption $(2)$ and $(3)$ are applied in the last inequality. Thus, it follows by \eqref{c-in-1} and \eqref{c-in-2} that
\bear \label{c-in-3}
\int_{T} |\nabla f|^2
& =& \int_{T} |\nabla F|^2 \\
& \geq& \frac{1}{4}\int_{T\cup S}F^2 - \int_S |\nabla F|^2 \nonumber\\
& \geq &\frac{1}{4}c - 4\delta c \quad \quad \emph{(by assumption $(1)$)} \nonumber\\
& = &\frac{1-16\delta}{4} c. \nonumber
\eear
Which completes the proof.
\end{proof}

%%%%%%%%%%%%%%%%%%%%%%%%%%%%%%%%%%%%%%%%%%%%%%%%%%%%%%%%%%%%%%%%%%%%%%%%%%%%%%%%%%%%

\section{Uniform gaps for eigenfunctions}\label{sec-gap}
Recall that the \emph{Cheeger isoperimetric constant} $h(X_g)$ is defined as
$$h(X_g):= \inf \frac{\length(\Gamma)}{\min \{\Vol(A_1),\Vol(A_2)\}}$$
where the infimum is taken over all smooth curves $\Gamma$ which divide $X_g$ into two pieces $A_1$ and $A_2$.

\begin{lemma}[Cheeger inequality, \cite{Che70}]\label{Cheeger}
Then
$$\lambda_1(X_g) \geq \frac{1}{4}h^2(X_g).$$
\end{lemma}

Let $\Gamma$ be a set of smooth curves dividing $X_g$ into two disjoint pieces $A_1$ and $A_2$. Then $\Gamma$ must be one of the following three cases:

$(a)$. $\Gamma$ contains a simple closed curve bounding a disk $D$ in $X_g$.

$(b)$. $\Gamma$ contains two simple closed curves $\tau$ and $\gamma$ which bounds a cylinder $T$ in $X_g$.

$(c)$. $\Gamma$ is not of type (a) and (b). That is, no two pairwise simple closed curves in $\Gamma$ are homotopic, and no simple closed curve in $\Gamma$ is homotopically trivial. In particular, $\length(\Gamma) \geq \mathcal L_1(X_g)$.

For cases $(a)$ and $(b)$, it follows by elementary isoperimetric inequalities that $\length(\Gamma) \geq \Vol (D)$ and  $\length(\Gamma) \geq \Vol (T)$. For case $(c)$, we have $\frac{\length(\Gamma)}{\min \{\Vol(A_1),\Vol(A_2)\}} \geq \frac{\mathcal L_1(X_g)}{\Vol(X_g)}$. Thus it follows by the Cheeger inequality that
\begin{equation}\label{l_1^2/g^2}
\lambda_1 (X_g) \geq \min\{\frac{1}{4},\frac{\mathcal L_1(X_g)^2}{4\Vol(X_g)^2}\}.
\end{equation}
One may see the proof of \cite[Proposition 9]{WX18} for more details on \eqref{l_1^2/g^2}.\\

In light of \eqref{l_1^2/g^2}, to prove $\lambda_1(X_g) \geq \min\{\frac{1}{4},c\frac{\mathcal L_1(X_g)}{\Vol(X_g)^2}\}$ we only need to consider the case that $\mathcal L_1(X_g)$ is small. The method in this article is motivated by \cite{DR86} of Dodziuk-Randol, which gave a different proof on the main results of Schoen-Yau-Wolpert in \cite{SWY80}.

For fixed small enough constant $\eps>0$ given by Lemma \ref{thick-thin} (for example $\eps=0.05$), we always assume \bear \mathcal L_1(X_g) \leq \eps.\eear In particular the modified thin part $\mathcal{B}$ defined in Section \ref{thick-thin decomp} is non-empty.

Let $\vph$ be an eigenfunction with respect to the first eigenvalue $\lambda_1$ on $X_g$ such that
$$\int_{X_g}\vph^2 = 1.$$

Denote the components of the modified thick part $\sA\subset X_g$ by $M_1,...,M_m$ where $m>0\in \Z$. For each integer $i \in [1,m]$ we set
$$\widehat{M_i} = \{p\in X_g ; \ \dist(p,M_i)\leq \eps\}.$$
When $\eps$ is small enough, these subsets $\widehat{M_1},...,\widehat{M_m}$ are still pairwise disjoint.

\begin{definition}
The \emph{oscillation} of $\vph$ on each component $M_i$, denoted by $\Osc(i)$, is defined as
\[\Osc(i):=\max_{x \in M_i}\vph(x)-\min_{x \in M_i}\vph(x).\]
\end{definition}

We have the following bound for $\sum_{i=1}^m \Osc(i)$ when $\lambda_1(X_g)$ is small.

\begin{proposition}\label{oscillation}
Let $X_g \in \sM_g$ be a hyperbolic surface with $\lambda_1(X_g) \leq \frac{1}{4}$, and $\vph$ be an eigenfunction with respect to $\lambda_1(X_g)$ with
$\int_{X_g}\vph^2 = 1$. Then there exists a constant $c=c(\eps)>0$ only depending on $\eps$ such that 
$$\sum_{i=1}^m \Osc(i) \leq c \sqrt{\Vol(X_g)\lambda_1(X_g)}.$$
\end{proposition}

\begin{proof}
On every component $M_{i_0}$ of $\sA$, we assume $p,q \in M_{i_0}$ with
\[\vph(p)=\max_{x\in M_{i_0}}\vph(x) \quad \emph{and} \quad \vph(q)=\min_{x\in M_{i_0}}\vph(x).\]

\noindent Consider a shortest path $\gamma: [0,l] \to M_{i_0}$ in $M_{i_0}$ connecting $p$ and $q$ with arc-parameter, where $l>0$ is the length of $\gamma$. Let $0=t_1<t_2<...<t_k=l$ be a partition of $[0,l]$ where $k>0\in \Z$ satisfying $t_{i+1}-t_{i}=\frac{1}{2}\eps$ for all $1\leq i\leq k-2$ and $t_k-t_{k-1}\leq \frac{1}{2}\eps$. Let $p_i = \gamma(t_i)$ and consider the embedding balls $B(p_i;\eps)$ centered at $p_i$ of radius $\eps$. By standard Sobolev embeddings \cite{Taylor-book} we know that
\begin{equation}
||\nabla \vph||_{L^\infty(B(p_i; \frac{\eps}{2}))} \leq c(\eps) \sum_{j=0}^N ||\Delta^j (d\vph)||_{L^2(B(p_i;\eps))} \nonumber
\end{equation}
for some integer $N>0$ (we remark here that the Sobolev embedding holds with a uniform constant because it follows by by Lemma \ref{thick-thin} that the injectivity radius satisfies that $\inj(p_i) \geq \eps$). Since $\vph$ is an eigenfunction and $\lambda_1(X_g) <\frac{1}{4}$,
\begin{align}\label{l2-up}
||\nabla \vph||_{L^\infty(B(p_i; \frac{\eps}{2}))}
&\leq c(\eps) \sum_{j=0}^N ||\Delta^j (d\vph)||_{L^2(B(p_i; \eps))} \\
&= c(\eps) \sum_{j=0}^N \lambda_1^j||d\vph||_{L^2(B(p_i; \eps))}\nonumber\\
&\leq \frac{4}{3}c(\eps) ||\nabla\vph||_{L^2(B(p_i; \eps))}. \nonumber
\end{align}
Such an estimation \eqref{l2-up} is standard. One may see \cite[Proposition 2.2]{Lip-Ste-18} for general estimations. Since $p_i$ and $p_{i+1}$ are both contained in the embedding ball $\overline{B(p_i; \frac{\eps}{2})}$, by \eqref{l2-up} we have
\begin{equation}\label{osc-2}
|\vph(p_i)-\vph(p_{i+1})| \leq \frac{2}{3}\eps c(\eps) ||\nabla\vph||_{L^2(B(p_i; \eps))}.
\end{equation}

\noindent Recall that $\widehat{M_{i_0}} = \{p\in X_g ; \ \dist(p,M_{i_0})\leq \eps\}$ and so $$B(p_i; \eps) \sbs \widehat{M_{i_0}}.$$

\noindent Suppose
$$B(p_i; \eps) \cap B(p_j; \eps) \neq \emptyset$$
for some $i\neq j$. By the triangle inequality we have $$\dist(p_i,p_j)< 2\eps.$$ Then it follows by Lemma \ref{thick-thin} that there exists a path $\gamma'\subset M_{i_0}$ connecting $p_i$ and $p_j$ of length $$\ell(\gamma')\leq 10\eps.$$ Since $\gamma$ is a shortest path and $t_{i+1}-t_{i}=\frac{1}{2}\eps$, we have $$|i-j|\leq 21.$$
Where since the last two values satisfies that $t_k-t_{k-1}\leq \frac{\eps}{2}$, we use $21$ instead of $20$ in the inequality above. So we have
\bear \label{osc-3}
B(p_i; \eps) \cap B(p_{i+r};\eps) = \emptyset,\ \ \emph{for all} \ r \geq 22.
\eear

\noindent Which implies that each point in $\widehat{M_{i_0}}$ can be only contained in at most $21$ embedding balls in $\{B(p_i; \eps)\}_{1\leq i\leq k}$. By Lemma \ref{thick-thin} we know that the injectivity radius satisfies that $\inj(x) \geq \eps$ for all $x\in M_{i_0}$. Thus, we have that number $k$ satisfies
\bear \label{osc-1}
k \leq \frac{21}{\Vol (B(\eps))} \Vol (\widehat{M_{i_0}})
\eear
where $\Vol (B(\eps))$, only depending on $\eps$, is the hyperbolic volume of a geodesic ball $B(\eps) \subset \H$ with radius $\eps$. Then it follow by \eqref{osc-2}, \eqref{osc-3}, \eqref{osc-1} and the Cauchy-Schwarz inequality that
\begin{align}
\Osc(i_0)
&= |\vph(p)-\vph(q)|  \\
&\leq \sum_{i=1}^{k-1} |\vph(p_i)-\vph(p_{i+1})|\nonumber \\
&\leq \frac{2}{3}\eps c(\eps) \sum_{i=1}^{k-1} ||\nabla\vph||_{L^2(B(p_i;\eps))}\nonumber \\
&= \frac{2}{3}\eps c(\eps) \sum_{i=1}^{k-1} \sqrt{\int_{B(p_i; \eps)} |\nabla\vph|^2 } \nonumber \\
&\leq \frac{2}{3}\eps c(\eps) \sqrt{k-1} \sqrt{\sum_{i=1}^{k-1}\int_{B(p_i;\eps)} |\nabla\vph|^2 } \nonumber \\
&\leq \frac{14\eps c(\eps)}{\sqrt{\Vol (B(\eps))}} \sqrt{\Vol (\widehat{M_{i_0}})} \sqrt{\int_{\widehat{M_{i_0}}} |\nabla\vph|^2 }. \nonumber
\end{align}

\noindent Since $i_0\in [1,m]$ is arbitrary, by taking a summation we get
\begin{align}
\sum_{i=1}^m \Osc(i)
&\leq \sum_{i=1}^m \frac{14\eps c(\eps)}{\sqrt{\Vol (B(\eps))}} \sqrt{\Vol (\widehat{M_{i}})} \sqrt{\int_{\widehat{M_{i}}} |\nabla\vph|^2}  \\
&\leq \frac{14\eps c(\eps)}{\sqrt{\Vol (B(\eps))}} \sqrt{\sum_{i=1}^m \Vol (\widehat{M_i})} \sqrt{\sum_{i=1}^m \int_{\widehat{M_i}} |\nabla\vph|^2 }\nonumber \\
&\leq \frac{14\eps c(\eps)}{\sqrt{\Vol (B(\eps))}} \sqrt{\Vol (X_g)} \sqrt{\int_{X_g} |\nabla\vph|^2 }\nonumber \\
&= \frac{14\eps c(\eps)}{\sqrt{\Vol (B(\eps))}} \sqrt{\Vol(X_g)\lambda_1(X_g)}. \nonumber
\end{align}

Then the conclusion follows by setting $c(\eps)=\frac{14\eps c(\eps)}{\sqrt{\Vol (B(\eps))}}$.
\end{proof}

Proposition \ref{oscillation} tells that the total oscillations of $\vph$ over components of $\sA$ is small. However, the following two results will tell that the oscillation of $\vph$ on $\sA$ is big if $\lambda_1(X_g)$ is small enough. This roughly tells that the total oscillations of $\vph$ over components of $X_g \setminus \sA$ is big in some sense.

\begin{lemma}\label{sign change}
Let $X_g \in \sM_g$ be a hyperbolic surface with $\lambda_1(X_g) \leq \frac{1}{4}$, and $\vph$ be the eigenfunction with $\Delta \vph+\lambda_1(X_g)\cdot \vph=0$. Then there exist two points $p_1\neq p_2 \in \sA$, where $\sA$ is defined in \eqref{def-sA}, such that
$$\vph(p_1)\cdot \vph(p_2) \leq 0.$$
\end{lemma}

\begin{proof}
Suppose for contradiction that $\vph < 0$ or $\vph>0$ (globally) on $\sA$. Without loss of generality we assume that $\vph>0$ on $\sA$; otherwise one may replace $-\vph$ by $\vph$. Since $\int_{X_g} \vph = 0$, there exists a collar $T\subset X_g$ such that $\min_{p\in T}\vph(p)<0$. By assumption we know that $\vph>0$ on the boundary of $T$. Thus, the nodal set $\{\vph=0\}$ bounds at least one non-empty subset $T'$ of $T\subset X_g$. By the analyticity of eigenfunction, one may assume that the boundary $\partial T'$ of $T'$ is smooth (\textsl{e.g.} see \cite{Cheng76}). Then by the Stokes' Theorem we have
$$\int_{T'} |\nabla \vph|^2 =\lambda_1(X_g) \cdot \int_{T'}  \vph^2.$$

\noindent Set
\[\widetilde{\vph}:=
\begin{cases}
\vph & \text{on}\ T', \\
0 & \text{on}\ T\setminus T' .
\end{cases}
\]
As $\widetilde{\vph}$ vanishes on $\partial T$, it follows by Lemma \ref{eigen collar1} that
$$\frac{1}{4} < \lambda_1(T) \leq \frac{\int_{T} |\nabla \widetilde{\vph}|^2}{\int_{T} |\widetilde{\vph}|^2} = \lambda_1(X_g).$$
Which is a contradiction.
\end{proof}

\begin{rem*}
Lemma \ref{sign change} implies that $$\max_{x\in \sA}\vph(x) - \min_{x\in \sA}\vph(x) \geq \sup_{x\in \sA}|\vph(x)|.$$
\end{rem*}

We next show that if $\lambda_1(X_g)$ is small enough, then the magnitude of the eigenfunction on $\sA$ has a uniform positive lower bound in term of the genus. More precisely,
\begin{proposition}\label{max thick part}
Let $X_g \in \sM_g$ be a hyperbolic surface with $$\lambda_1(X_g) \leq \frac{1}{1000}\frac{1}{\Vol(X_g)}.$$ And let $\vph$ be eigenfunction with respect to $\lambda_1(X_g)$ with
$\int_{X_g}\vph^2 = 1$. Then
$$\sup_{x\in \sA}|\vph(x)|\geq \frac{1}{32\sqrt{\Vol(X_g)}}.$$
\end{proposition}

\begin{proof}
Suppose for contradiction that
$$\sup_{x\in \sA}|\vph(x)| < s\frac{1}{\sqrt{\Vol(X_g)}}$$
where $$s=\frac{1}{32}.$$ Then we have
$$\int_{\sA} \vph^2 \leq s^2 \frac{\Vol(\sA)}{\Vol(X_g)}.$$

\noindent Since $\int_{X_g} \vph^2=1$, $X_g$ has non-empty modified thin part $\sB$ defined in \eqref{def-sB}. Thus one may assume that $\sB$ consists of disjoint collars $T_1,...,T_k$ where $k>0 \in \Z$. Then we have

\beqar
\sum_{i=1}^k \int_{T_i} \vph^2
&=&1- \int_{\sA} \vph^2 \nonumber \\
&\geq& 1-s^2 \nonumber \\
&\geq& (1-s^2)\sum_{i=1}^k \frac{\Vol(T_i)}{\Vol(X_g)}.\nonumber
\eeqar

\noindent Thus for certain component $T$ of $\sB$, we have
$$\int_T \vph^2 \geq (1-s^2)\Vol(T)\frac{1}{\Vol(X_g)}.$$

\noindent Recall that the shell $S$ of collar $T$ is defined as
$$S=\{p\in X_g; \ 0<\dist(p,T)\leq 1\} \subset \sA$$
and Lemma \ref{thick-thin} tells that
$$\frac{1}{2}\leq\Vol(T)\leq 4 \quad \emph{and} \quad \frac{1}{2}\leq\Vol(S)\leq 4.$$
So we have
\beqar
\int_T \vph^2 &\geq& \frac{1-s^2}{2} \frac{1}{\Vol(X_g)},\\
\int_S \vph^2 &\leq&  s^2\frac{\Vol(S)}{\Vol(X_g)} \leq 4s^2 \frac{1}{\Vol(X_g)},\\
\int_S |\nabla \vph|^2 &\leq& \int_{X_g} |\nabla \vph|^2 = \lambda_1(X_g) \leq \frac{1}{1000}\frac{1}{\Vol(X_g)}.
\eeqar
Note that $s^2 = \frac{1}{1024}$. We apply Lemma \ref{eigen collar2} for the case that $c=\frac{1-s^2}{2} \frac{1}{\Vol(X_g)}$ and $\delta=\frac{1}{64}$ to get
\beqar
\int_{T} |\nabla\vph|^2
&\geq& \frac{3}{16} \frac{1-\frac{1}{1024}}{2} \frac{1}{\Vol(X_g)} \nonumber \\
&\geq& \frac{1}{16} \frac{1}{\Vol(X_g)}. \nonumber
\eeqar
Then we have
$$\lambda_1(X_g) =\int_{X_g} |\nabla\vph|^2 \geq \frac{1}{16} \frac{1}{\Vol(X_g)}$$
which contradicts our assumption
\[\lambda_1(X_g) \leq \frac{1}{1000}\frac{1}{\Vol(X_g)}.\]
The proof is complete.
\end{proof}

Recall \eqref{l_1^2/g^2} says that
\begin{equation}
\lambda_1 (X_g) \geq \min\{\frac{1}{4},\frac{\mathcal L_1(X_g)^2}{4\Vol(X_g)^2}\} \nonumber
\end{equation}

In order to prove Theorem \ref{mt-1}, i.e., $$\lambda_1(X_g) \geq K \frac{\mathcal L_1(X_g)}{\Vol(X_g)^2}$$ where $K>0$ is a uniform constant independent of $g$. In light of \eqref{l_1^2/g^2} it suffices to consider the case that $\mathcal L_1(X_g)$ is small. Now we assume that \bear \mathcal L_1(X_g) \leq \eps \eear and
\begin{equation}\label{lambda assumption}
\lambda_1(X_g) \leq \frac{1}{1000\eps} \frac{\mathcal L_1(X_g)}{\Vol(X_g)^2}.
\end{equation}
Since $\Vol(X_g)=4\pi(g-1)$, by \eqref{lambda assumption} we know that $$\lambda_1(X_g) \leq \frac{1}{4}.$$ So first it follows by Proposition \ref{oscillation} that
\begin{eqnarray}\label{osc sum}
\sum_{i=1}^m \Osc(i) &\leq& c(\eps) \sqrt{\Vol(X_g)\lambda_1(X_g)}\\
&\leq & c(\eps)\sqrt{\frac{\mathcal L_1(X_g)}{1000\eps}} \frac{1}{\sqrt{\Vol(X_g)}}.\nonumber
\end{eqnarray}

\noindent Since $\mathcal{L}_1(X_g)\leq \eps$ and $\Vol(X_g)=4\pi(g-1)>1$, by \eqref{lambda assumption} we know that $$\lambda_1(X_g) \leq \min\{\frac{1}{4},\frac{1}{1000}\frac{1}{\Vol(X_g)}\}.$$ Then it follows by the remark following Lemma \ref{sign change} and Proposition \ref{max thick part} that
\begin{equation}\label{max-min-1}
\max_{x\in \sA}\vph(x) - \min_{x\in \sA}\vph(x) \geq \frac{1}{32\sqrt{\Vol(X_g)}}.
\end{equation}

Now we have the following result.
\begin{proposition}\label{eige-gap}
Let $\eps>0$ in Lemma \ref{thick-thin} and $X_g$ be a hyperbolic surface of genus $g$ with 
\begin{equation}\label{how small l_1-1}
c(\eps)\sqrt{\frac{\mathcal L_1(X_g)}{1000\eps}} < \frac{1}{64} \quad \emph{and} \quad \lambda_1(X_g) \leq \frac{1}{1000\eps} \frac{\mathcal L_1(X_g)}{\Vol(X_g)^2} \quad \emph{and} \quad \mathcal L_1(X_g) \leq \eps.\nonumber
\end{equation}
Then we have
\begin{equation}
|\max_{x\in \sA}\vph(x) - \min_{x\in \sA}\vph(x)| - \sum_{i=1}^m \Osc(i) \geq \frac{1}{64\sqrt{\Vol(X_g)}}. \nonumber
\end{equation}
\end{proposition}
\bp
It clearly follows by \eqref{osc sum} and \eqref{max-min-1}.
\ep

\noindent Proposition \ref{eige-gap} in particular implies that the complement $X_g-\sA$ of $\sA$ is non-empty if the assumptions in the proposition hold.

%%%%%%%%%%%%%%%%%%%%%%%%%%%%%%%%%%%%
\section{Proof of Theorem \ref{mt-1}} \label{sec-mt-1}
In this section we prove Theorem \ref{mt-1}. We first make some necessary preparations.

Recall that $\{M_i\}_{1\leq i \leq m}$ are the components of $\sA$. For each $i \in [1,m]$ we let $a_i, b_i \in \R$ such that 
$$a_i=\min_{x\in M_i} \vph(x) \quad \emph{and} \quad b_i=\max_{x\in M_i} \vph(x).$$
That is, $Im (\vph|_{M_i})=[a_i,b_i]$. So for each $i \in [1,m]$,
$$\Osc(i)=b_i-a_i.$$

For each component of $X_g\setminus \sA$, it is a collar whose two boundary curves are contained in two components denoted by $M_i$ and $M_j$ of $\sA$ ($M_i$ may be the same as $M_j$). There may exist multiple collars bounded by $M_i$ and $M_j$. We denote these collars by $T_{ij}^1,...,T_{ij}^{\theta_{ij}}$. Set
\begin{equation}
\delta_{ij}^\theta = \min_{(x,x^*)\in \Gamma_1\times \Gamma_2} |\vph(x)-\vph(x^*)| \nonumber
\end{equation}
where $\Gamma_1,\Gamma_2$ are the two boundaries of $T_{ij}^\theta$ and $x^*\in \Gamma_2$ is the reflection of $x\in\Gamma_1$ through the center closed geodesic $\gamma_{ij}^\theta$. By Lemma \ref{eigen collar3} we have
\begin{equation}
\int_{T_{ij}^\theta} |\nabla \vph|^2 \geq \frac{r_{ij}^\theta}{4}(\delta_{ij}^\theta)^2 \nonumber
\end{equation}
where $r_{ij}^\theta$ is the length of center closed geodesic of collar $T_{ij}^\theta$. Set
\begin{equation} \label{delta-ij}
\delta_{ij} =
\begin{cases}
0 & \text{if}\ [a_i,b_i]\cap[a_j,b_j]\neq \emptyset, \\
\dist([a_i,b_i],[a_j,b_j]) & \text{if}\ [a_i,b_i]\cap[a_j,b_j]= \emptyset.
\end{cases}
\end{equation}
Then
\bear
\delta_{ij}^\theta \geq \delta_{ij}.
\eear
Let $\{T_{ij}^\theta\}$ be all the components of $X_g\setminus \sA$. Putting the inequalities above together we get
\begin{align}\label{before the last}
\lambda_1(X_g)
&= \int_{X_g} |\nabla \vph|^2  \\
&\geq \sum_{T_{ij}^\theta} \int_{T_{ij}^\theta} |\nabla \vph|^2 \nonumber \\
&\geq \sum_{T_{ij}^\theta} \frac{r_{ij}^\theta}{4}(\delta_{ij})^2 \nonumber \\
&\geq \frac{1}{4} \frac{1}{\sharp(\emph{components of} \ X_g\setminus \sA)} \left(\sum_{T_{ij}^\theta} \sqrt{r_{ij}^\theta}\delta_{ij}\right)^2 \nonumber \\
&\geq \frac{1}{12(g-1)} \left(\sum_{T_{ij}^\theta} \sqrt{r_{ij}^\theta}\delta_{ij}\right)^2. \nonumber
\end{align}

\noindent Next we will bound the quantity $\sum_{T_{ij}^\theta} \sqrt{r_{ij}^\theta}\delta_{ij}$ from below in terms of $\mathcal{L}_1(X_g)$ and $\Vol(X_g)$. As a function, the summation $\sum_{T_{ij}^\theta} \sqrt{r_{ij}^\theta}\delta_{ij}$ is linear with respect to the intervals $[a_i,b_i]$. So it achieves its minimum at some extremal case, which may be related to $\mathcal{L}_1(X_g)$. We first prove the following elementary property.

\begin{lemma}\label{linear lemma}
Let $n>0\in \Z$ and assume $\{I_i=[a_i,b_i]\}_{1\leq i \leq n}$ are $n$ closed intervals with increasing order
\begin{equation}\label{relation of intervals}
a_1 \leq b_1 \leq a_2 \leq b_2 \leq a_3...< a_n \leq b_n.
\end{equation}
Let $\delta=\sum_{i=1}^n |b_i-a_i|$ be the total lengths of all the intervals. Then for any collection of nonnegative numbers $\{\alpha_{ij}\}_{1\leq i<j \leq n}$, there exists an integer $K_0$ with $1\leq K_0 < n$ such that
\begin{equation}\label{elemq-1}
\sum_{i<j}\alpha_{ij}\dist(I_i,I_j) \geq (b_n-a_1-\sum_{i=1}^n |b_i-a_i|)\sum_{1\leq i\leq K_0 < j\leq n} \alpha_{ij}. 
\end{equation}
\end{lemma}

\begin{proof}
We prove it by induction on $n$.

If $n=1$, both sides of \eqref{elemq-1} are equal to $0$.

If $n=2$, \eqref{elemq-1} holds by letting $K_0=1$.

Assume that \eqref{elemq-1} holds for $k\leq (n-1)$.

Now for $k=n$, we first regard
\begin{equation}
L_n(a_1,b_1,...,a_n,b_n)=\sum_{i<j}\alpha_{ij}\dist(I_i,I_j) =\sum_{i<j}\alpha_{ij}(a_j-b_i) \nonumber
\end{equation}
as a function $(a_1,b_1,...,a_n,b_n)$. Set
\begin{equation}
L(t)=L_n(a_1,b_1,a_2+t,b_2+t,a_3,b_3,...,a_n,b_n). \nonumber
\end{equation}
Clearly $L(t)$ is linear with respect to $t$ on its domain which is the one preserving the relation (\ref{relation of intervals}). So $L(t)$ takes its minimum on the boundaries, i.e., when
$$a_2+t=b_1\ \ \text{or}\ \ b_2+t=a_3.$$

\noindent So we have
\begin{equation}\label{elem-2}
L_n(a_1,b_1;...) \geq \min
\left\{
  \begin{array}{c}
    L_n(a_1,b_1,b_1,(b_1+b_2-a_2),a_3,b_3,...,a_n,b_n), \\
    L_n(a_1,b_1,(a_3+a_2-b_2),a_3,a_3,b_3,...,a_n,b_n) \\
  \end{array}
\right\}.
\end{equation}
Without loss of generality one may assume that $L_n(a_1,b_1,b_1,(b_1+b_2-a_2),...,a_n,b_n)$ is the smaller one. Then we have
\bear\label{elem-3}
&& L_n(a_1,b_1,b_1,(b_1+b_2-a_2),...,a_n,b_n)= \sum_{i\geq 3}\alpha_{1i}(a_i-b_1)\\
&&+\sum_{i\geq 3}\alpha_{2i}(a_i-(b_1+b_2-a_2))+\sum_{3\leq i <j \leq n}\alpha_{ij}(a_j-b_i) \nonumber \\
&\geq&\sum_{i\geq 3}(\alpha_{1i}+\alpha_{2i})(a_i-(b_1+b_2-a_2))+  \sum_{3\leq i <j \leq n}\alpha_{ij}(a_j-b_i)  \nonumber
\eear
Set \[I_1'=[a_1,b_1+b_2-a_2] \quad \emph{and} \quad \ I_i'=[a_{i+1},b_{i+1}]  \ (2\leq i \leq (n-1))\]
and
\[\alpha_{1i}'=\alpha_{1(i+1)}+\alpha_{2(i+1)}\quad \emph{and} \quad \alpha_{ij}'=\alpha_{(i+1)(j+1)} \ (2\leq i <j\leq (n-1)).\]
Then we have
\beqar
\sum_{i\geq 3}(\alpha_{1i}+\alpha_{2i})(a_i-b_1-b_2+a_2)+  \sum_{3\leq i <j \leq n}\alpha_{ij}(a_j-b_i)=\sum_{1\leq i<j\leq n-1}\alpha'_{ij}\dist(I_i',I_j').
\eeqar

\noindent By our induction assumption on $k=(n-1)$, we know that there exists an integer $K_1\in [1,n-1]$ such that
\bear\label{elem-4}
\sum_{1\leq i<j\leq (n-1)}\alpha'_{ij}\dist(I_i',I_j') &\geq& (b_n-a_1-\sum_{i=1}^n |b_i-a_i|)\sum_{1\leq i\leq K_1 < j\leq (n-1)} \alpha'_{ij}\\
&=& (b_n-a_1-\sum_{i=1}^n |b_i-a_i|)\sum_{1\leq i\leq (K_1+1) < j\leq n} \alpha_{ij}. \nonumber
\eear
\noindent Set $K_0=K_1+1$. Then the conclusion follows by \eqref{elem-2}, \eqref{elem-3} and \eqref{elem-4}.
\end{proof}

Now we return to estimate the quantity $\sum_{T_{ij}^\theta} \sqrt{r_{ij}^\theta}\delta_{ij}$ in \eqref{before the last}.
\begin{proposition}\label{linear method}
\begin{align}
\sum_{T_{ij}^\theta} \sqrt{r_{ij}^\theta}\delta_{ij}
&\geq \sqrt{\mathcal L_1(X_g)}\left(\max_{i,j}|b_i-a_j| - \sum_{i=1}^m |b_i-a_i|  \right) \nonumber \\
&= \sqrt{\mathcal L_1(X_g)}\left(|\max_{x\in \sA}\vph(x) - \min_{x\in \sA}\vph(x)| - \sum_{i=1}^m \Osc(i)  \right) \nonumber
\end{align}
where $[a_i,b_i] = Im(\vph|_{M_i})$.
\end{proposition}

\begin{proof}
Assume
$$\bigcup_{i=1}^m [a_i,b_i] = \bigcup_{k=1}^N [e_k,f_k]$$
where the $[e_k,f_k]$ are disjoint and
\begin{equation}
\min_{x\in \sA}\vph(x)= e_1 \leq f_1 < e_2 \leq f_2 <...< e_k \leq f_k = \max_{x\in \sA}\vph(x). \nonumber
\end{equation}
Since the $[e_k,f_k]$ are disjoint, for each $i \in [1,m]$ there exists a unique $k_i \in [1,N]$ such that
$$Im(\vph|_{M_i}) = [a_i,b_i] \sbs [e_{k_i},f_{k_i}].$$
For $i\neq j$, we set
\[ \delta'_{ij}:=\dist([e_{k_i},f_{k_i}],[e_{k_j},f_{k_j}]).\]
Let $\delta_{ij}$ be the constants given in \eqref{delta-ij}. Clearly we have
\begin{equation}
\delta_{ij} \geq \delta'_{ij} \quad \emph{for all $i\neq j$}.
\end{equation}
Which implies that
\bear \label{sum-l-1}
\sum_{T_{ij}^\theta} \sqrt{r_{ij}^\theta}\delta_{ij}\geq \sum_{T_{ij}^\theta} \sqrt{r_{ij}^\theta}\delta'_{ij}.
\eear
We rewrite the quantity $\sum_{T_{ij}^\theta} \sqrt{r_{ij}^\theta}\delta'_{ij}$ as
\bear \label{sum-l-1.5}
\sum_{T_{ij}^\theta} \sqrt{r_{ij}^\theta}\delta'_{ij}=\sum_{p<q} \left(\sum_{\ k_i=p,\ k_j=q}\sqrt{r_{ij}^\theta}\right) \dist([e_p,f_p],[e_q,f_q]).
\eear

\noindent Then it follows by Lemma \ref{linear lemma} that there exists an integer $K_0 \in [1, N-1]$ such that
\begin{align} \label{sum-l-2}
&\sum_{p<q} \left(\sum_{\ k_i=p,\ k_j=q}\sqrt{r_{ij}^\theta}\right) \dist([e_p,f_p],[e_q,f_q]) \\
&\geq (f_N-e_1-\sum_{k=1}^N |f_k-e_k|) \sum_{1\leq k_i \leq K_0 < k_j\leq N} \sqrt{r_{ij}^\theta}  \nonumber \\
&\geq \left(\sum_{1\leq k_i \leq K_0 < k_j\leq N} \sqrt{r_{ij}^\theta}\right) \left(\max_{i,j}|b_i-a_j| - \sum_{i=1}^m |b_i-a_i|\right). \nonumber
\end{align}

\noindent Note that the unions of all the center closed geodesics of collars $T_{ij}^\theta$, which satisfy $1\leq k_i \leq K_0 < k_j\leq N$,  divide $X_g$ into at least two components: one contains a $M_i$ with $k_i\leq K_0$; another one contains a $M_j$ with $k_j > K_0$. So by the definition of $\mathcal L_1(X_g)$ we have
\begin{equation} 
\sum_{1\leq k_i \leq K_0 < k_j\leq N} r_{ij}^\theta \geq \mathcal L_1(X_g)
\end{equation}
implying that
\bear \label{sum-l-3}
\sum_{1\leq k_i \leq K_0 < k_j\leq N} \sqrt{r_{ij}^\theta}\geq \sqrt{\sum_{1\leq k_i \leq K_0 < k_j\leq N} r_{ij}^\theta}\geq \sqrt{\mathcal L_1(X_g)}.
\eear

Then it follows by \eqref{sum-l-1}, \eqref{sum-l-1.5}, \eqref{sum-l-2} and \eqref{sum-l-3} that
\begin{align}
\sum_{T_{ij}^\theta} \sqrt{r_{ij}^\theta}\delta_{ij}\geq \sqrt{\mathcal L_1(X_g)} \left(\max_{i,j}|b_i-a_j|-\sum_{i=1}^m |b_i-a_i|\right)
\end{align}
which completes the proof.
\end{proof}

Now we are ready to prove Theorem \ref{mt-1}.
\bp [Proof of Theorem \ref{mt-1}]
Recall that \eqref{l_1^2/g^2} says that
\[\lambda_1 (X_g) \geq \min\{\frac{1}{4},\frac{\mathcal L_1(X_g)^2}{4\Vol(X_g)^2}\}.\]
Let $\eps>0$ be the uniform constant in Lemma \ref{thick-thin}. First we assume that  
\begin{equation}\label{how small l_1-1-1}
c(\eps)\sqrt{\frac{\mathcal L_1(X_g)}{1000\eps}} < \frac{1}{64} \quad \text{and} \quad \mathcal{L}_1 (X_g) \leq \eps,
\end{equation}
otherwise we have
\bear
\frac{\mathcal L_1(X_g)^2}{\Vol(X_g)^2} \geq\frac{1000\eps}{64^2 c(\eps)^2} \frac{\mathcal L_1(X_g)}{\Vol(X_g)^2} 
\ \text{or} \
\frac{\mathcal L_1(X_g)^2}{\Vol(X_g)^2} \geq \eps  \frac{\mathcal L_1(X_g)}{\Vol(X_g)^2}
\eear
which together with \eqref{l_1^2/g^2} complete the proof. Now we also assume that
\bear \label{eige-v}
\lambda_1(X_g) \leq \frac{1}{1000\eps} \frac{\mathcal L_1(X_g)}{\Vol(X_g)^2},
\eear
otherwise we are done. By our assumption \eqref{how small l_1-1-1} and \eqref{eige-v}, we apply Proposition \ref{eige-gap} to get
\begin{equation}
|\max_{x\in \sA}\vph(x) - \min_{x\in \sA}\vph(x)| - \sum_{i=1}^m \Osc(i) \geq \frac{1}{64\sqrt{\Vol(X_g)}}. 
\end{equation}

\noindent Thus, by combining (\ref{before the last}) and Proposition \ref{linear method} we have
\begin{align}
\lambda_1(X_g)
&\geq \frac{1}{12(g-1)} \left(\sum_{T_{ij}^\theta} \sqrt{r_{ij}^\theta}\delta_{ij}\right)^2 \ \\
&\geq \frac{\mathcal L_1(X_g)}{12(g-1)}\left(|\max_{x\in A}\vph(x) - \min_{x\in A}\vph(x)| - \sum_{i=1}^m \Osc(i) \right)^2. \nonumber \\
&\geq \frac{\mathcal L_1(X_g)}{49152} \frac{1}{(g-1)\cdot \Vol(X_g)}. \nonumber
\end{align}

\noindent Recall that $\Vol(X_g)=4\pi(g-1)$. Therefore, the discussions above imply that
\begin{align}
\lambda_1(X_g)
&\geq \min
\left\{
\begin{array}{c}
	\frac{1}{4}, \eps \cdot \frac{\mathcal L_1(X_g)}{4\Vol(X_g)^2}, \frac{1000\eps}{64^2 c(\eps)^2}\cdot \frac{\mathcal L_1(X_g)}{4\Vol(X_g)^2},  \\
	\frac{1}{1000\eps} \frac{\mathcal L_1(X_g)}{\Vol(X_g)^2}, \frac{\mathcal L_1(X_g)}{49152} \frac{1}{(g-1)\cdot \Vol(X_g)} \\
\end{array}
\right\} \\
&\geq \min\{\frac{1}{4}, c(\eps)\cdot \frac{\mathcal L_1(X_g)}{g^2} \}  \nonumber
\end{align}
where $c(\eps)>0$ is a uniform constant only depending on $\eps$. By Lemma \ref{ub-L1} we know that \[ \mathcal{L}_1(X_g) \leq c'\cdot g\] for some uniform constant $c'>0$. So for large enough $g$, $ c(\eps)\cdot \frac{\mathcal L_1(X_g)}{g^2} <\frac{1}{4}$. Thus, we get
\begin{equation}
\lambda_1(X_g) \geq K_1 \cdot \frac{\mathcal L_1(X_g)}{g^2}
\end{equation}
for some uniform constant $K_1>0$ independent of $g$. The proof is complete.
\ep

%%%%%%%%%%%%%%%%%%%%%%%%%%%%%%%%%%%%%%%%%%%%%%%%%%%%%%%%%%%%%%%%%%%%%%%%%%%%%%%%%%%%

\section{An optimal example}\label{sec-opti}

In this section, we use the results in \cite{WX18} to explain that the lower bound of $\lambda_1(X_g)$ in Theorem \ref{mt-1} is optimal as $g\to \infty$. That is, for all genus $g\geq 2$, there exists some Riemann surfaces $\mathcal X_g$ of genus $g$ such that
\begin{equation}
\lambda_1(\mathcal X_g) \leq K_2 \frac{\mathcal L_1(\mathcal X_g)}{g^2} \nonumber
\end{equation}
for some uniform constant $K_2>0$, especially when $\mathcal L_1(\mathcal X_g)$ is arbitrarily small.

The following construction of $\mathcal X_g$ is given in \cite{WX18}.

Recall that a pair of pants is a compact Riemann surface of $0$ genus with $3$ boundary closed geodesics. The complex structure is uniquely determined by the lengths of the three boundary closed geodesics.

For a fixed
\begin{equation}
\ell < 2\arcsinh 1, \nonumber
\end{equation}
consider the pair of pants $\mathcal{P}_\ell$ whose boundary curves all have length equal to $\ell$. We construct $\mathcal X_g$ by gluing $3g-3$ pants $\mathcal P_\ell$ along boundary loops (with arbitrary twist parameters) as shown in figure \ref{pic surface}.

\begin{figure}[h]
\begin{center}
\begin{tikzpicture}[scale=1]

\draw (0,0.3) ..controls +(-0.5,0)and +(0.5,0).. +(-1,0.5)  (0,-0.3) ..controls +(-0.5,0)and +(0.5,0).. +(-1,-0.5);
\draw (-2,0.3) ..controls +(0.5,0)and +(-0.5,0).. +(1,0.5)  (-2,-0.3) ..controls +(0.5,0)and +(-0.5,0).. +(1,-0.5);

\draw (-2,0.3) ..controls +(-0.5,0)and +(0.5,0).. +(-1,0.5)  (-2,-0.3) ..controls +(-0.5,0)and +(0.5,0).. +(-1,-0.5);
\draw (-4,0.3) ..controls +(0.5,0)and +(-0.5,0).. +(1,0.5)  (-4,-0.3) ..controls +(0.5,0)and +(-0.5,0).. +(1,-0.5);

\draw (-4,0.3) ..controls +(-0.5,0)and +(0.5,0).. +(-1,0.5)  (-4,-0.3) ..controls +(-0.5,0)and +(0.5,0).. +(-1,-0.5);
\draw (-5,0.8) ..controls +(-0.7,0)and +(0,0.3).. +(-1,-0.8) (-5,-0.8) ..controls +(-0.7,0)and +(0,-0.3).. +(-1,0.8);

\draw (2,0.3) ..controls +(0.5,0)and +(-0.5,0).. +(1,0.5)  (2,-0.3) ..controls +(0.5,0)and +(-0.5,0).. +(1,-0.5);
\draw (3,0.8) ..controls +(0.7,0)and +(0,0.3).. +(1,-0.8) (3,-0.8) ..controls +(0.7,0)and +(0,-0.3).. +(1,0.8);

\draw (0,0) +(0,0.3)--+(0.1,0.3) +(0,-0.3)--+(0.1,-0.3);
\draw (2,0) +(0,0.3)--+(-0.1,0.3) +(0,-0.3)--+(-0.1,-0.3);

\draw[dashed] (0.3,0)--(1.7,0);

\draw (-1,-0.2) ..controls +(0.1,0).. +(0.3,0.2)  +(0,0) ..controls +(-0.1,0).. +(-0.3,0.2)
+(0,0.4) ..controls +(0.1,0).. +(0.3,0.2)  +(0,0.4) ..controls +(-0.1,0).. +(-0.3,0.2)
+(0.3,0.2)--+(0.4,0.3)  +(-0.3,0.2)--+(-0.4,0.3);

\draw (-3,-0.2) ..controls +(0.1,0).. +(0.3,0.2)  +(0,0) ..controls +(-0.1,0).. +(-0.3,0.2)
+(0,0.4) ..controls +(0.1,0).. +(0.3,0.2)  +(0,0.4) ..controls +(-0.1,0).. +(-0.3,0.2)
+(0.3,0.2)--+(0.4,0.3)  +(-0.3,0.2)--+(-0.4,0.3);

\draw (-5,-0.2) ..controls +(0.1,0).. +(0.3,0.2)  +(0,0) ..controls +(-0.1,0).. +(-0.3,0.2)
+(0,0.4) ..controls +(0.1,0).. +(0.3,0.2)  +(0,0.4) ..controls +(-0.1,0).. +(-0.3,0.2)
+(0.3,0.2)--+(0.4,0.3)  +(-0.3,0.2)--+(-0.4,0.3);

\draw (3,-0.2) ..controls +(0.1,0).. +(0.3,0.2)  +(0,0) ..controls +(-0.1,0).. +(-0.3,0.2)
+(0,0.4) ..controls +(0.1,0).. +(0.3,0.2)  +(0,0.4) ..controls +(-0.1,0).. +(-0.3,0.2)
+(0.3,0.2)--+(0.4,0.3)  +(-0.3,0.2)--+(-0.4,0.3);

\draw[dashed][thick] (0,0) +(0,-0.3) ..controls +(0.1,0.3).. +(0,0.3);
\draw[thick] (0,0) +(0,-0.3) ..controls +(-0.1,0.3).. +(0,0.3);

\draw[dashed][thick] (2,0) +(0,-0.3) ..controls +(0.1,0.3).. +(0,0.3);
\draw[thick] (2,0) +(0,-0.3) ..controls +(-0.1,0.3).. +(0,0.3);

\draw[dashed][thick] (-2,0) +(0,-0.3) ..controls +(0.1,0.3).. +(0,0.3);
\draw[thick] (-2,0) +(0,-0.3) ..controls +(-0.1,0.3).. +(0,0.3);

\draw[dashed][thick] (-4,0) +(0,-0.3) ..controls +(0.1,0.3).. +(0,0.3);
\draw[thick] (-4,0) +(0,-0.3) ..controls +(-0.1,0.3).. +(0,0.3);

\draw[dashed] (-1,0.5) +(0,-0.3) ..controls +(0.1,0.3).. +(0,0.3);
\draw (-1,0.5) +(0,-0.3) ..controls +(-0.1,0.3).. +(0,0.3);

\draw[dashed] (-1,-0.5) +(0,-0.3) ..controls +(0.1,0.3).. +(0,0.3);
\draw (-1,-0.5) +(0,-0.3) ..controls +(-0.1,0.3).. +(0,0.3);

\draw[dashed] (-3,0.5) +(0,-0.3) ..controls +(0.1,0.3).. +(0,0.3);
\draw (-3,0.5) +(0,-0.3) ..controls +(-0.1,0.3).. +(0,0.3);

\draw[dashed] (-3,-0.5) +(0,-0.3) ..controls +(0.1,0.3).. +(0,0.3);
\draw (-3,-0.5) +(0,-0.3) ..controls +(-0.1,0.3).. +(0,0.3);

\draw (-6,0) ..controls +(0.35,-0.1).. +(0.7,0);
\draw[dashed] (-6,0) ..controls +(0.35,0.1).. +(0.7,0);

\draw (4,0) ..controls +(-0.35,-0.1).. +(-0.7,0);
\draw[dashed] (4,0) ..controls +(-0.35,0.1).. +(-0.7,0);

\end{tikzpicture}
\end{center}
\caption{Surface $\mathcal{X}_g$} \label{pic surface}
\end{figure}
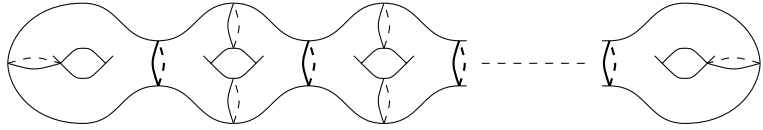

For any closed geodesic $\gamma \subset \mathcal X_g$, the curve $\gamma$ is either one of the boundary closed geodesic of certain $\mathcal{P}_\ell$ or must intersect at least one of the boundary closed geodesic of certain $\mathcal{P}_\ell$. Recall that the Collar Lemma \ref{collar} implies that for any two intersected closed geodesics, then at least one of them have length larger than $2\arcsinh 1$. So if $\gamma$ is a systolic curve of $\mathcal X_g$, then $\gamma$ must be one of the boundary curve of $\mathcal{P}_\ell$. Thus, the systole of $X_g$ and $\mathcal L_1(\mathcal X_g)$ are both equal to $\ell$. That is,
\begin{equation}
\mathcal L_1(\mathcal X_g) = \sys(\mathcal X_g) = \ell. \nonumber
\end{equation}

\noindent It was proved \cite[Proposition 15 or Remark 16]{WX18}  that
\begin{equation}
\lambda_1(\mathcal X_g) \leq \frac{\beta(\ell)}{g^2}. 
\end{equation}
\noindent Where for the case that $\ell< 2\arcsinh 1$ and $\mathcal X_g$ constructed as above, $$\beta(\ell) \leq K_2\cdot \ell$$ for some universal constant $K_2>0$. So we have
\begin{equation}
\lambda_1(\mathcal X_g) \leq K_2\frac{\mathcal L_1(\mathcal X_g)}{g^2} 
\end{equation}
which tells that the lower bound of $\lambda_1(X_g)$ in Theorem \ref{mt-1} is optimal either as 
$g\to \infty$ or as $\mathcal{L}_1(X_g)\to 0$.

\bibliographystyle{plain}
\bibliography{ref}

\begin{thebibliography}{10}

\bibitem{BMM16}
Werner Ballmann, Henrik Matthiesen, and Sugata Mondal.
\newblock Small eigenvalues of closed surfaces.
\newblock {\em J. Differential Geom.}, 103(1):1--13, 2016.

\bibitem{BMM17}
Werner Ballmann, Henrik Matthiesen, and Sugata Mondal.
\newblock On the analytic systole of {R}iemannian surfaces of finite type.
\newblock {\em Geom. Funct. Anal.}, 27(5):1070--1105, 2017.

\bibitem{Bers-c}
Lipman Bers.
\newblock An inequality for {R}iemann surfaces.
\newblock In {\em Differential geometry and complex analysis}, pages 87--93.
  Springer, Berlin, 1985.

\bibitem{Bus77}
Peter Buser.
\newblock Riemannsche {F}l\"achen mit {E}igenwerten in {$(0,$} {$1/4)$}.
\newblock {\em Comment. Math. Helv.}, 52(1):25--34, 1977.

\bibitem{Buser10}
Peter Buser.
\newblock {\em Geometry and spectra of compact {R}iemann surfaces}, volume 106
  of {\em Progress in Mathematics}.
\newblock Birkh\"auser Boston, Inc., Boston, MA, 1992.

\bibitem{Chavel}
Isaac Chavel.
\newblock {\em Eigenvalues in {R}iemannian geometry}, volume 115 of {\em Pure
  and Applied Mathematics}.
\newblock Academic Press, Inc., Orlando, FL, 1984.
\newblock Including a chapter by Burton Randol, With an appendix by Jozef
  Dodziuk.

\bibitem{Che70}
Jeff Cheeger.
\newblock A lower bound for the smallest eigenvalue of the {L}aplacian.
\newblock In {\em Problems in analysis ({P}apers dedicated to {S}alomon
  {B}ochner, 1969)}, pages 195--199. 1970.

\bibitem{Cheng76}
Shiu~Yuen Cheng.
\newblock Eigenfunctions and nodal sets.
\newblock {\em Comment. Math. Helv.}, 51(1):43--55, 1976.

\bibitem{DPRS85}
Jozef Dodziuk, Thea Pignataro, Burton Randol, and Dennis Sullivan.
\newblock Estimating small eigenvalues of {R}iemann surfaces.
\newblock In {\em The legacy of {S}onya {K}ovalevskaya ({C}ambridge, {M}ass.,
  and {A}mherst, {M}ass., 1985)}, volume~64 of {\em Contemp. Math.}, pages
  93--121.

\bibitem{DR86}
Jozef Dodziuk and Burton Randol.
\newblock Lower bounds for {$\lambda_1$} on a finite-volume hyperbolic
  manifold.
\newblock {\em J. Differential Geom.}, 24(1):133--139, 1986.

\bibitem{Lip-Ste-18}
Michael Lipnowski and Mark Stern.
\newblock Geometry of the smallest 1-form {L}aplacian eigenvalue on hyperbolic
  manifolds.
\newblock {\em Geom. Funct. Anal.}, 28(6):1717--1755, 2018.

\bibitem{Mon14}
Sugata Mondal.
\newblock Systole and {$\lambda_{2g-2}$} of closed hyperbolic surfaces of genus
  {$g$}.
\newblock {\em Enseign. Math.}, 60(1-2):3--24, 2014.

\bibitem{NWX20}
Xin {Nie}, Yunhui {Wu}, and Yuhao {Xue}.
\newblock {Large genus asymptotics for lengths of separating closed geodesics
  on random surfaces}.
\newblock {\em arXiv e-prints}, page arXiv:2009.07538, September 2020.

\bibitem{OR09}
Jean-Pierre Otal and Eulalio Rosas.
\newblock Pour toute surface hyperbolique de genre {$g,\ \lambda_{2g-2}>1/4$}.
\newblock {\em Duke Math. J.}, 150(1):101--115, 2009.

\bibitem{SWY80}
R.~Schoen, S.~Wolpert, and S.~T. Yau.
\newblock Geometric bounds on the low eigenvalues of a compact surface.
\newblock In {\em Geometry of the {L}aplace operator ({P}roc. {S}ympos. {P}ure
  {M}ath., {U}niv. {H}awaii, {H}onolulu, {H}awaii, 1979)}, Proc. Sympos. Pure
  Math., XXXVI, pages 279--285. Amer. Math. Soc., Providence, R.I., 1980.

\bibitem{Taylor-book}
Michael~E. Taylor.
\newblock {\em Partial differential equations {I}. {B}asic theory}, volume 115
  of {\em Applied Mathematical Sciences}.
\newblock Springer, New York, second edition, 2011.

\bibitem{WX18}
Yunhui {Wu} and Yuhao {Xue}.
\newblock {Small eigenvalues of closed Riemann surfaces for large genus}.
\newblock {\em arXiv e-prints}, page arXiv:1809.07449, Sep 2018.

\end{thebibliography}

\end{document}